\documentclass[12pt]{article}

\usepackage{amsthm, amsmath, amssymb,graphicx}

\numberwithin{equation}{section}

\theoremstyle{plain}	     
\newtheorem{thm}{Theorem}[section] 

\newtheorem{lem}[thm]{Lemma}

\theoremstyle{definition}

\theoremstyle{remark}

\newcommand{\vp}{\varphi}

\newcommand{\sn}{\operatorname{sn}}
\newcommand{\cn}{\operatorname{cn}}
\newcommand{\dn}{\operatorname{dn}}

\begin{document}

\title{Some double-angle formulas related to a generalized lemniscate function
\footnote{The work was supported by JSPS KAKENHI Grant Number 17K05336.}}
\author{Shingo Takeuchi \\
Department of Mathematical Sciences\\
Shibaura Institute of Technology
\thanks{307 Fukasaku, Minuma-ku,
Saitama-shi, Saitama 337-8570, Japan. \endgraf
{\it E-mail address\/}: shingo@shibaura-it.ac.jp \endgraf
}}


\maketitle

\begin{abstract}
In this paper we will establish some double-angle formulas
related to the inverse function of $\int_0^x dt/\sqrt{1-t^6}$.
This function appears in Ramanujan's Notebooks
and is regarded as a generalized version of 
the lemniscate function.

\end{abstract}



\section{Introduction}

Let $1<p,\ q<\infty$ and 
$$F_{p,q}(x):=\int_0^x \frac{dt}{(1-t^q)^{1/p}}, \quad x \in [0,1].$$
We will denote by $\sin_{p,q}$ the inverse function of $F_{p,q}$, i.e.,
$$\sin_{p,q}{x}:=F_{p,q}^{-1}(x).$$
Clealy, $\sin_{p,q}{x}$ is an increasing function in $[0,\pi_{p,q}/2]$ to $[0,1]$,
where
$$\pi_{p,q}:=2F_{p,q}(1)=2\int_0^1 \frac{dt}{(1-t^q)^{1/p}}.$$
We extend $\sin_{p,q}{x}$ to $(\pi_{p,q}/2,\pi_{p,q}]$ by $\sin_{p,q}{(\pi_{p,q}-x)}$
and to the whole real line $\mathbb{R}$ as the odd $2\pi_{p,q}$-periodic 
continuation of the function. 
Since $\sin_{p,q}{x} \in C^1(\mathbb{R})$,
we also define $\cos_{p,q}{x}$ by $\cos_{p,q}{x}:=(d/dx)(\sin_{p,q}{x})$.
Then, it follows that 
$$|\cos_{p,q}{x}|^p+|\sin_{p,q}{x}|^q=1.$$
In case $(p,q)=(2,2)$, it is obvious that $\sin_{p,q}{x},\ \cos_{p,q}{x}$ 
and $\pi_{p,q}$ are reduced to the ordinary $\sin{x},\ \cos{x}$ and $\pi$,
respectively. 
This is a reason why these functions and the constant are called
\textit{generalized trigonometric functions} (with parameter $(p,q)$)
and the \textit{generalized $\pi$}, respectively. 

The generalized trigonometric functions are well studied in the context of 
nonlinear differential equations (see \cite{KT2019} and the references given there). 
Suppose that $u$ is a solution of 
the initial value problem of $p$-Laplacian
$$-(|u'|^{p-2}u')'=\frac{(p-1)q}{p} |u|^{q-2}u, \quad u(0)=0,\ u'(0)=1,$$
which is reduced to the equation $-u''=u$ of simple harmonic motion for
$u=\sin{x}$ in case $(p,q)=(2,2)$.  
Then, 
$$\frac{d}{dx}(|u'|^p+|u|^q)=\left(\frac{p}{p-1}(|u'|^{p-2}u')'+q|u|^{q-2}u\right)u'=0.$$
Therefore, $|u'|^p+|u|^q=1$, hence it is reasonable to define $u$ as 
a generalized sine function and $u'$ as a generalized cosine function.
Indeed, it is possible to show that $u$ coincides with $\sin_{p,q}$ defined as above.
The generalized trigonometric functions are often applied to
the eigenvalue problem of $p$-Laplacian.

Now, we are interested in finding double-angle formulas for 
generalized trigonometric functions.
It is possible to discuss addition formulas for these functions, 
but for simplicity we will not develop this point here.

We have known the double-angle formulas 
of $\sin_{2,q},\ \sin_{q^*,q}$ and $\sin_{q^*,2}$
for $q=2,3,4$ except for $\sin_{3/2,2}$, where $q^*:=q/(q-1)$
 (Table \ref{tab:(p,q)}). For the detail for each formula,
we refer the reader to \cite{ST2020}
(After having proved the formula for $\sin_{2,3}$ 
in the co-authored paper \cite{ST2020}, the author noticed that 
the formula has already been obtained as $\vp(2s)$ by Cox and Shurman
\cite[p.697]{CS2005}).
It is worth pointing out that Lemma \ref{lem:maf} (resp. Lemma \ref{lem:duality}) 
below connects the parameter $(2,q)$ to $(q^*,q)$ (resp. $(q^*,2)$) 
and yields the possibility 
to obtain the other formula from one formula.
Indeed, in this way, the formulas of $\sin_{4/3,4}$ and $\sin_{4/3,2}$ 
follow from that of $\sin_{2,4}$ 
(\cite[Subsection 3.1]{T2016b} and \cite[Theorem 1.1]{ST2020}, respectively),
and the formula of $\sin_{2,3}$ follows from that of $\sin_{3/2,3}$ (\cite[Theorem 1.2]{ST2020}).
Nevertheless, the parameter $(3/2,2)$ is still open
because of difficulty of the inverse problem corresponding to \eqref{eq:f(2x)}.

In this paper, we wish to investigate the double-angle formula 
of the function $\sin_{2,6}{x}$, whose inverse function is defined as
$$\sin_{2,6}^{-1}{x}=\int_0^x \frac{dt}{\sqrt{1-t^6}}.$$
The function $\sin_{2,6}{x}$ appears as the inverse of $H(v)$ 
in Ramanujan's Notebooks \cite[p.246]{Be1994}
and is regarded as a generalized version 
of the lemniscate function $\sin_{2,4}{x}$. 
For the function,
Shinohara \cite{Sh2017} gives the novel double-angle formula  
\begin{equation}
\label{eq:(p,q)=(2,6)}
\sin_{2,6}{(2x)}=\frac{2\sin_{2,6}{x}\cos_{2,6}{x}}{\sqrt{1+8\sin_{2,6}^6{x}}},
\quad x \in [0,\pi_{2,6}/2].
\end{equation}
In fact, he found \eqref{eq:(p,q)=(2,6)} in ``trial and error calculations'' 
(according to private communication), but instead we will 
make a heuristic proof of \eqref{eq:(p,q)=(2,6)} in Section \ref{sec:(2,6)}.
Moreover, as mentioned above, we can show the following
counterparts of \eqref{eq:(p,q)=(2,6)} for $\sin_{6/5,6}$ and $\sin_{6/5,2}$,
respectively.

\begin{thm}
\label{thm:(p,q)=(6/5,6)}
Let $(p,q)=(6/5,6)$. Then, for $x \in [0,\pi_{6/5,6}/4]$,
\begin{equation*}
\sin_{6/5,6}{(2x)}
=\frac{2^{1/6}\sin_{6/5,6}{x}\cos_{6/5,6}^{1/5}{x}\left(3+\sqrt{1+32\sin_{6/5,6}^6{x}
\cos_{6/5,6}^{6/5}{x}}\right)^{1/2}}{\left(1+32\sin_{6/5,6}^6{x}\cos_{6/5,6}^{6/5}{x}\right)^{1/4}
\left(1+\sqrt{1+32\sin_{6/5,6}^6{x}\cos_{6/5,6}^{6/5}{x}}\right)^{1/6}}.
\end{equation*}
\end{thm}

\begin{thm}
\label{thm:(p,q)=(6/5,2)}
Let $(p,q)=(6/5,2)$. Then, for $x \in [0,\pi_{6/5,2}/2]$,
\begin{equation*}
\sin_{6/5,2}{(2x)}
=\sqrt{1-\left(\frac{9-8\sin_{6/5,2}^2{x}-4\sin_{6/5,2}^2{x}\cos_{6/5,2}^{2/5}{x}}
{9-8\sin_{6/5,2}^2{x}+8\sin_{6/5,2}^2{x}\cos_{6/5,2}^{2/5}{x}}\right)^3}.
\end{equation*}
\end{thm}

\begin{table}[h]
\begin{center}
\begin{tabular}{|c|l|l|l|}
\noalign{\hrule height0.8pt}
$q$ & $(q^*,2)$ & $(2,q)$ & $(q^*,q)$ \\
\hline
$2$ & $(2,2)$ by Abu al-Wafa' & $(2,2)$ by Abu al-Wafa' & $(2,2)$ by Abu al-Wafa' \\
$3$ & $(3/2,2)$ open & $(2,3)$ by Cox-Shurman & $(3/2,3)$ by Dixon \\
$4$ & $(4/3,2)$ by Sato-Takeuchi & $(2,4)$ by Fagnano & $(4/3,4)$ by Edmunds et al. \\
$6$ & $(6/5,2)$ \textbf{Theorem \ref{thm:(p,q)=(6/5,2)}} & $(2,6)$ by Shinohara & $(6/5,6)$ \textbf{Theorem \ref{thm:(p,q)=(6/5,6)}} \\
\noalign{\hrule height0.8pt}
\end{tabular}
\end{center}
\caption{The parameters for which the double-angle formulas have been obtained.}
\label{tab:(p,q)}
\end{table}


\section{Proof of \eqref{eq:(p,q)=(2,6)}}
\label{sec:(2,6)}

The change of variable $s=t^2$ leads to the representation
$$\sin_{2,6}^{-1}{x}=\frac{1}{2}\int_0^{x^2}\frac{ds}{\sqrt{s(1-s^3)}},\quad 0 \leq x \leq 1.$$
The furthermore change of valiable
$$\cn{u}=\frac{1-(\sqrt{3}+1)s}{1+(\sqrt{3}-1)s}, \quad k^2=\frac{2-\sqrt{3}}{4}$$
gives
\begin{align*}
\sin_{2,6}^{-1}{x}
&=\frac{1}{2}\int_0^{\cn^{-1}{\phi(x)}}
\frac{((\sqrt{3}+1)+(\sqrt{3}-1)\cn{u})^2}{2\cdot 3^{3/4}\sn{u}\dn{u}}
\cdot \frac{2\sqrt{3}\sn{u}\dn{u}}{((\sqrt{3}+1)+(\sqrt{3}-1)\cn{u})^2}\,du\\
&=\frac{1}{2\cdot 3^{1/4}}\int_0^{\cn^{-1}{\phi(x)}}\,du\\
&=\frac{1}{2\cdot 3^{1/4}} \cn^{-1}{\phi(x)},
\end{align*}
where $\sn{u}=\sn{(u,k)},\ \cn{u}=\cn{(u,k)}$ and $\dn{u}=\dn{(u,k)}$ are the
Jacobian elliptic functions (see e.g. \cite[Chapter XXII]{WW1927} for more details), and
$$\phi(x)=\frac{1-(\sqrt{3}+1)x^2}{1+(\sqrt{3}-1)x^2}.$$
Thus,
$$\sin_{2,6}{u}=\phi^{-1}{(\cn{(2\cdot 3^{1/4}u)})}, \quad 0 \leq u \leq \pi_{2,6}/2=K/(3^{1/4}),$$
where $K=K(k)$ is the complete elliptic integral of the first kind and
$$\phi^{-1}(x)=\sqrt{\frac{1-x}{(\sqrt{3}+1)+(\sqrt{3}-1)x}}.$$

Now, we use the addition formula of $\cn$. 
For $u,\ v,\ u \pm v \in [0,K/(3^{1/4})]$,
\begin{align*}
\sin_{2,6}{(u \pm v)}
&=\phi^{-1}{(\cn{(\tilde{u} \pm \tilde{v})})}\\
&=\phi^{-1}{\left(\frac{\cn{\tilde{u}}\cn{\tilde{v}} \mp \sn{\tilde{u}}\sn{\tilde{v}}\dn{\tilde{u}}\dn{\tilde{v}}}
{1-k^2\sn^2{\tilde{u}}\sn^2{\tilde{v}}}\right)}
\end{align*}
where $\tilde{u}:=2\cdot 3^{1/4}u$ and $\tilde{v}:=2\cdot 3^{1/4}v$.
Recall that $\sn^2{x}+\cn^2{x}=1$ and $k^2\sn^2{x}+\dn^2{x}=1$; then the last equality gives
\begin{multline}
\label{eq:addition26}
\sin_{2,6}{(u \pm v)}\\
=\phi^{-1}{\left(\frac{\phi(U)\phi(V) \mp \sqrt{(1-\phi(U)^2)(1-\phi(V)^2)(1-k^2(1-\phi(U)^2))(1-k^2(1-\phi(V)^2))}}
{1-k^2(1-\phi(U)^2)(1-\phi(V)^2)}\right)},
\end{multline}
where $U:=\sin_{2,6}{u}$ and $V:=\sin_{2,6}{v}$.

With $u=v$ and the observation that 
\begin{gather*}
1-\phi(U)^2=\frac{4\sqrt{3}U^2(1-U^2)}{(1+(\sqrt{3}-1)U^2)^2},\\
1-k^2(1-\phi(U)^2)=\frac{1+U^2+U^4}{(1+(\sqrt{3}-1)U^2)^2},
\end{gather*}
this implies that 
\begin{align*}
\sin_{2,6}^{-1}{(2u)}
&=\phi^{-1}{\left(\frac{\phi(U)^2-(1-\phi(U)^2)(1-k^2(1-\phi(U)^2))}{1-k^2(1-\phi(U)^2)^2}\right)}\\
&=\phi^{-1}{\left(\frac{1-4(\sqrt{3}+1)U^2+8U^6+4(\sqrt{3}+1)U^8}{1+4(\sqrt{3}-1)U^2+8U^6-4(\sqrt{3}-1)U^8}\right)}.
\end{align*}
Routine simplification now results in the formula
$$\sin_{2,6}^{-1}{(2u)}
=\frac{2U\sqrt{1-U^6}}{\sqrt{1+8U^6}}
=\frac{2\sin_{2,6}{u}\cos_{2,6}{u}}{\sqrt{1+8\sin_{2,6}^6{u}}},$$
and the proof is complete.
\qed


\section{Proofs of theorems}

To prove Theorem \ref{thm:(p,q)=(6/5,6)}, we use the following multiple-angle formulas. 

\begin{lem}[\cite{T2016b}]
\label{lem:maf}
Let $1<q<\infty$ 
and $q^*:=q/(q-1)$. If $x \in [0,\pi_{2,q}/(2^{2/q})]=[0,\pi_{q^*,q}/2]$, then
\begin{align}
\sin_{2,q}{(2^{2/q}x)}&=2^{2/q}\sin_{q^*,q}{x}\cos_{q^*,q}^{q^*-1}{x}, \label{eq:mafs} \\
\cos_{2,q}{(2^{2/q} x)}&=\cos_{q^*,q}^{q^*}{x}-\sin_{q^*,q}^q{x} \notag \\
&=1-2\sin_{q^*,q}^q{x}=2\cos_{q^*,q}^{q^*}{x}-1. \label{eq:mafc}
\end{align}
\end{lem}

\begin{proof}[Proof of Theorem \ref{thm:(p,q)=(6/5,6)}]
Let $x \in [0,\pi_{6/5,6}/4]$.
Applying \eqref{eq:mafs} of Lemma \ref{lem:maf} in case $q=6$
with $x$ replaced by $2x \in [0,\pi_{6/5,6}/2]$, we get
\begin{equation}
\label{eq:cubic}
\sin_{2,6}{(2\cdot2^{1/3}x)}=2^{1/3}\sin_{6/5,6}{(2x)}(1-\sin_{6/5,6}^6{(2x)})^{1/6}.
\end{equation}

First, we consider the case 
$$0 \leq x <\frac{\pi_{6/5,6}}{8}.$$
Then, since $0 \leq 2\sin_{6/5,6}^6{(2x)}<1$ by \cite[Lemma 2.1]{T2016b}, 
the equation \eqref{eq:cubic} gives
$$2\sin_{6/5,6}^6{(2x)}
=1-\sqrt{1-\sin_{2,6}^6{(2\cdot2^{1/3}x)}}.$$
Set $S=S(x):=\sin_{2,6}{(2^{1/3}x)}$. Using the double-angle formula 
\eqref{eq:(p,q)=(2,6)} for $\sin_{2,6}{x}$, we have
\begin{align*}
2\sin_{6/5,6}^6{(2x)}
&=1-\sqrt{1-\left(\frac{2S\sqrt{1-S^6}}{\sqrt{1+8S^6}}\right)^6}\\
&=1-\frac{\sqrt{1-40S^6+384S^{12}+320S^{18}+64S^{24}}}{(1+8S^6)^{3/2}}\\
&=1-\frac{|1-20S^6-8S^{12}|}{(1+8S^6)^{3/2}}.
\end{align*}
Since $0 \leq S^6 <\sin_{2,6}^6{(\pi_{2,6}/4)}=(3\sqrt{3}-5)/4$, 
evaluated by \eqref{eq:(p,q)=(2,6)},
we see that $1-20S^6-8S^{12}>0$. 
Thus, 
\begin{align}
2\sin_{6/5,6}^6{(2x)}
&=1-\frac{1-20S^6-8S^{12}}{(1+8S^6)^{3/2}} \notag \\
&=\frac{(\sqrt{1+8S^6}-1)(\sqrt{1+8S^6}+3)^3}{8(1+8S^6)^{3/2}} \notag \\
&=\frac{S^6(3+\sqrt{1+8S^6})^3}{(1+8S^6)^{3/2}(1+\sqrt{1+8S^6})}.\label{eq:case1}
\end{align}
Therefore, by \eqref{eq:mafs}, 
$$\sin_{6/5,6}{(2x)}
=\frac{2^{1/6}\sin_{6/5,6}{x}\cos_{6/5,6}^{1/5}{x}\left(3+\sqrt{1+32\sin_{6/5,6}^6{x}
\cos_{6/5,6}^{6/5}{x}}\right)^{1/2}}{\left(1+32\sin_{6/5,6}^6{x}\cos_{6/5,6}^{6/5}{x}\right)^{1/4}
\left(1+\sqrt{1+32\sin_{6/5,6}^6{x}\cos_{6/5,6}^{6/5}{x}}\right)^{1/6}}.$$

In the remaining case 
$$\frac{\pi_{6/5,6}}{8} 
\leq x \leq \frac{\pi_{6/5,6}}{4},$$
it follows easily that $1 \leq 2\sin_{6/5,6}^6{(2x)}<2$ and $1-20S^6-8S^{12} \leq 0$,
hence we obtain \eqref{eq:case1} again.
The proof is complete.
\end{proof}

%

To show Theorem \ref{thm:(p,q)=(6/5,2)}, the following lemma is useful.

\begin{lem}[\cite{EGL2012}, \cite{KT2019}]
\label{lem:duality}
Let $1<p,\ q <\infty$. For $x \in [0,2]$,
\begin{gather*}
q\pi_{p,q}=p^*\pi_{q^*,p^*},\\
\sin_{p,q}{\left(\frac{\pi_{p,q}}{2}x\right)}
=\cos_{q^*,p^*}^{q^*-1}{\left(\frac{\pi_{q^*,p^*}}{2}(1-x)\right)}.
\end{gather*}
\end{lem}

\begin{proof}[Proof of Theorem \ref{thm:(p,q)=(6/5,2)}]
Let $x \in [0,\pi_{6/5,2}/2]$. Then,
since $4x/\pi_{6/5,2} \in [0,2]$, it follows from Lemma \ref{lem:duality} that
$$\sin_{6/5,2}{(2x)}
=\cos_{2,6}{\left(\frac{\pi_{2,6}}{2}\left(1-\frac{4x}{\pi_{6/5,2}}\right)\right)}
=\cos_{2,6}{\left(\frac{\pi_{2,6}}{2}-\frac{2x}{3}\right)}.$$
Thus,
\begin{equation}
\label{eq:sin2x}
\sin_{6/5,2}{2x}=\sqrt{1-\sin_{2,6}^6{\left(\frac{\pi_{2,6}}{2}-\frac{2x}{3}\right)}}.
\end{equation}
Function $\sin_{2,6}$ has the addition formula \eqref{eq:addition26}.
Letting $u=\pi_{2,6}/2$ and $v=2x/3$, we have
\begin{equation}
\label{eq:atlem}
\sin_{2,6}{\left(\frac{\pi_{2,6}}{2}-\frac{2x}{3}\right)}
=\phi^{-1}{\left(-\phi(V)\right)}=\sqrt{\frac{1-V^2}{1+2V^2}},
\end{equation}
where $V:=\sin_{2,6}{(2x/3)}$.
Applying \eqref{eq:atlem} to the right-hand side of \eqref{eq:sin2x}, we obtain
$$\sin_{6/5,2}{2x}=\sqrt{1-\left(\frac{1-\sin_{2,6}^2{(2x/3)}}{1+2\sin_{2,6}^2{(2x/3)}}\right)^3}.$$

Let $f(x):=\sin_{6/5,2}{x}$ and $g(x):=\sin_{2,6}{(2x/3)}$. Then, 
\begin{equation}
\label{eq:f(2x)}
f(2x)=\sqrt{1-\left(\frac{1-g(x)^2}{1+2g(x)^2}\right)^3}.
\end{equation}
Therefore, it is easy to see that 
\begin{align}
\label{eq:+}
g(x)=\sqrt{\frac{1-(1-f(2x)^2)^{1/3}}{1+2(1-f(2x)^2)^{1/3}}}.
\end{align}

On the other hand, by \eqref{eq:(p,q)=(2,6)} with $x$ replaced with $x/2$, 
we see that $g(x)$ satisfies 
\begin{equation*}
g(x)=\frac{2g(x/2)\sqrt{1-g(x/2)^6}}{\sqrt{1+8g(x/2)^6}}.
\end{equation*}
Applying \eqref{eq:+} with $x$ replaced with $x/2$ to the right-hand side, we obtain
\begin{equation}
\label{eq:g(x)}
g(x)=\frac{2f(x)(1-f(x)^2)^{1/6}}{\sqrt{9-8f(x)^2}}.
\end{equation}

Subtituting \eqref{eq:g(x)} into \eqref{eq:f(2x)},
we can express $f(2x)$ in terms of $f(x)$, i.e.,
$$f(2x)=\sqrt{1-\left(\frac{9-8f(x)^2-4f(x)^2(1-f(x)^2)^{1/3}}{9-8f(x)^2+8f(x)^2(1-f(x)^2)^{1/3}}\right)^3}.$$
Since $1-f(x)^2=\cos_{6/5,2}^{6/5}{x}$, the proof is complete.
\end{proof}



\end{document}